%% file: real_main.tex
\documentclass[12pt]{amsart}

\usepackage{lmodern} 
\usepackage{setspace} 

\usepackage{amsmath,amsthm,amsfonts,amssymb,mathrsfs}
\usepackage{graphicx}
\usepackage{graphics,xcolor}
\usepackage{tikz-cd}
\usepackage[all]{xy}

\usepackage{hyphenat}
\usepackage{natbib}

\newcommand{\RR}{\mathbb{R}}
\newcommand{\CC}{\mathbb{C}}
\newcommand{\NN}{\mathbb{N}}

\newcommand{\HH}{\mathbb{H}}

\newcommand{\Ss}{\mathbb{S}}

\DeclareMathOperator{\upboxdim}{\overline{\mathrm{dim}_B}}

\DeclareMathOperator{\packdim}{\mathrm{dim}_P}
\DeclareMathOperator{\hffdim}{\mathrm{dim}_H}

\newcommand{\boundangle}{\mathscr{G}_{\ge \delta}}
\newcommand{\simpends}{\mathscr{G}_{AS}}
\newcommand{\notboundtri}{\mathscr{G}^c_\Delta}

\newtheorem{claim}{Claim}
\newtheorem{thm}{Theorem}

\newtheorem{prop}{Proposition}[subsection]
\newtheorem{cor}[prop]{Corollary}
\newtheorem{lem}[prop]{Lemma}

\newtheorem*{thm*}{Theorem}
\newtheorem*{prop*}{Proposition}

\theoremstyle{definition}
\newtheorem{defi}[prop]{Definition}
\newtheorem{ex}[prop]{Example}

\theoremstyle{remark}
\newtheorem{rmk}[prop]{Remark}

\begin{document}

\title[Geodesics with bounded self-intersection angle]{On the Hausdorff Dimension of geodesics with bounded self-intersection angle}
\author{Joaquín Lejtreger}
\date{\today}
\address{
Sorbonne Université and Université Paris Cité, CNRS, IMJ-PRG, F-75005 Paris, France.}
\email{lejtreger@imj-prg.fr}

\begin{abstract}
    We expand a result of Birman–Series \cite{birman1985geodesics} by proving that the set of geodesics whose self-intersection angles are bounded from below has Hausdorff dimension zero. In addition, we show that the set of geodesics that do not bound a geodesic triangle also has Hausdorff dimension zero.
\end{abstract}

\maketitle

\input{statements}
\input{construction_radalla}

\input{simple_ends}
\input{no_triangles}

\bibliographystyle{alpha}   
\bibliography{references}   

\end{document}

%% file: statements.tex
\section{Introduction}

Throughout this paper, we denote by $S$ a closed hyperbolic surface. A hyperbolic structure on $S$ gives rise to an identification of the universal cover of $S$ with $\HH^2$, and an identification of the fundamental group $\pi_1(S)$ with a subgroup of $PSL_2(\RR)$: the orientation-preserving isometries of $\HH^2$.

Complete oriented geodesics in $\HH^2$ are in one-to-one correspondence with pairs of different points in the visual boundary $\partial\HH^2$. We denote the set of such pairs by $\mathscr{G} = \partial\HH^2 \times \partial \HH^2 \setminus \Delta$. 

Identifying $\partial \HH^2$ with a subset of $\CC$, we can endow $\mathscr{G}$ with a metric by adding the Euclidean distances of the endpoints in $\partial \HH^2$. It is H\"older equivalent to the visual metric on $\partial \HH^2$, and is related with how long geodesics fellow travel on $\HH^2$. For a subset $A$ of $\mathscr{G}$, we let $\hffdim A$ be the Hausdorff dimension with respect to this metric.

In this paper we will study typical geodesics in terms of Hausdorff dimension by considering an approach inspired by Birman-Series \cite{birman1985geodesics}.

\begin{defi}\label{def:tipicas}
   Let $A \subseteq \mathscr{G}$ be $\pi_1(S)$-invariant. We say that $A$ is \textit{typical} if $$\hffdim(\mathscr{G} \setminus A) = 0.$$
\end{defi}

Note that, while Hausdorff dimension of $A$ is a bi-Lipschitz invariant, the property in Definition \ref{def:tipicas} is a H\"older invariant. We will also study the Hausdorff dimension of the image of complements of typical sets in the surface $S$.


If $\alpha, \beta$ are two geodesics that intersect at $p$ we denote the unoriented angle of intersection as $\angle_p(\alpha, \beta) \in [0, \pi/2]$. Given $\alpha \in \mathscr{G}$ we define 
$$a(\alpha) = \inf_{g \in \pi_1(S)}\{\angle_{\alpha \cap g\alpha} (\alpha, g\alpha)\},$$
the smallest self-intersection angle of the projection of $\alpha$ to $S$ (elements $g \in \pi_1(S)$ with $\alpha \cap g\alpha = \emptyset$ are not considered).



We will study the set of geodesics whose self-intersection angles are uniformly bounded below: for $\delta > 0$ we let
\begin{equation}\label{def:bounded-below}
 \boundangle = \left\{\alpha \in \mathscr{G} : a(\alpha) \ge \delta \right\}.
\end{equation}

Observe that this is a $\pi_1(S)$-invariant set. Thus, by an abuse of notation, we will also consider $\boundangle$ as a subset of geodesics in $S$.

\begin{ex}\label{ex:boundangle}
   Let $\gamma$ and $\beta$ be two simple, intersecting geodesics, as in Figure \ref{fig:boundangle}, and let $\tilde{\gamma}$ and $\tilde{\beta}$ be intersecting lifts in the universal cover. Let $\tilde{\alpha}$ be a geodesic such that $\tilde{\alpha}(-\infty) = \tilde{\gamma}(-\infty)$ and $\tilde{\alpha}(+\infty) = \tilde{\beta}(+\infty)$. Then $\tilde{\alpha}$ is a geodesic in $\boundangle$ for some $\delta > 0$, but $\alpha$ has an infinite number of self-intersections.
\end{ex}

We now state the main theorem of the paper concerning this set.

\begin{thm}\label{teo:el-wan}
 For all $\delta > 0$, we have $\hffdim \boundangle = 0$. In particular, the typical geodesic has arbitrarily small self-intersection angles. Moreover, the image of $\boundangle$ as a subset of $S$ has Hausdorff dimension 1.
\end{thm}

\begin{figure}[ht]
   \centering
   \def\svgwidth{0.8\columnwidth}
   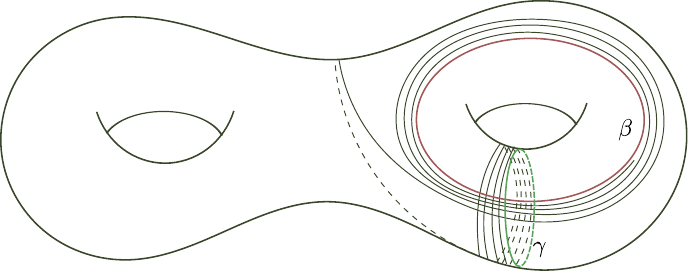
   \caption{An example of a geodesic $\alpha$ with infinite self-intersections and $a(\alpha) > 0$.}
   \label{fig:boundangle}
 \end{figure}

The main ingredient of the proof is the following proposition, which reduces the problem to the study of a special class of geodesics. 

   \begin{prop*}[Proposition \ref{prop:simple-ends}]
      Let $\alpha \in \boundangle$. There exists a finite cover $p: \tilde{S} \to S$ and a curve $\tilde{\alpha}$ in $\tilde{S}$ such that $p(\tilde{\alpha}) = \alpha$ and $\tilde{\alpha}$ is asymptotically simple.
   \end{prop*}

Asymptotically simple geodesics, defined in \ref{def:ends-simply}, are geodesics that are simple (i.e. no self-intersections) from a certain point onwards and backwards. 
To prove this proposition we encode the self-intersection data of $\alpha$ into a finite combinatorial object, called a \emph{band complex}, in the spirit of radallas, in the sense of Erlandsson-Souto \cite{erlandsson2016counting}. This, in turn, enables us to construct a finite cover in which these intersections are resolved.
From the proposition, the proof of Theorem \ref{teo:el-wan} essentially follows from the fact that these geodesics have Hausdorff dimension 0. This is not a straightforward consequence of Birman-Series \cite{birman1985geodesics}, which states that the set of simple geodesics have Hausdorff dimension 0.

Our second theorem follows from Theorem \ref{teo:el-wan}. Let $\alpha$ be a geodesic in $S$. We say that $\alpha$ \textit{bounds a triangle} if there exist lifts $\alpha_1, \alpha_2, \alpha_3$ to $\HH^2$ which intersect pairwise.
Equivalently, $\alpha$ bounds a triangle if there exists $L > 0$ such that a connected component of $S \setminus \alpha|_{[-L, L]}$ is a triangle with geodesic boundary.

Let $\mathscr{G}_\Delta$ be the set of geodesics that bound a triangle.

\begin{thm}\label{teo:triangulos}
The set $\mathscr{G}_\Delta$ is typical. Moreover, the image of the complement of $\mathscr{G}_\Delta$ in $S$ has Hausdorff dimension $1$.
\end{thm}

In the paper \cite{athreya2021local}, the authors investigate the statistical behavior of a random geodesic according to the Liouville measure. In particular, the authors study the polygons determined by arcs arising from the intersection between a random geodesic and a small ball $B(p, \epsilon)$ centered at a point $p \in S$. We emphasize that this is a measure theoretic viewpoint and as such differs from ours.

\subsection{Previous results}

We now provide some previous results that motivated Theorems \ref{teo:el-wan} and \ref{teo:triangulos}.

A foundational result in this direction is due to Birman-Series \cite{birman1985geodesics}, who prove that the set of geodesics that intersect themselves an infinite number of times is typical. In their paper, they also prove that the image of those geodesics on the surface has Hausdorff dimension 1.

A generalization of this result was proved by Sapir in \cite{sapir2016birman}. Given a complete geodesic $\alpha: \RR \to S$, let $\iota(\alpha_L)$ be the set of self-intersection points of $\alpha|_{[-L,L]}$. Now let $f: \RR_+ \to \RR_+$ be a function and define the set
$$\mathscr{G}(f) = \left\{ \alpha \in \mathscr{G}: \iota(\alpha_L) < f(L) \text{ for all } L \in \RR_+ \right\}.$$
Sapir proved the following theorem:

\begin{thm*}[Sapir \cite{sapir2016birman}]
   If $f(L) = o(L^2)$ then the image of $\mathscr{G}(f)$ in $S$ has Hausdorff dimension 1.
\end{thm*}

The author's techniques also allow to prove that the set $\mathscr{G}(f)$ itself has Hausdorff dimension 0. This is to say, the intersections of a typical geodesic grow quadratically in length. Observe that the geodesic described in \ref{ex:boundangle} is in $\boundangle$, but the growth of self-intersections is not of $o(L^2)$, showing that our result is not comprised in Sapir's theorem.

\subsection*{Acknowledgments}
I am grateful to Andrés Sambarino for his advice and guidance in the preparation of this paper, and to Viveka Erlandsson for suggesting a key idea for the proof of Theorem \ref{teo:triangulos}. I also thank Odylo Costa for offering helpful corrections. This work has received funding from the European Union’s Horizon 2020 research and innovation program under the Marie Sk\l{}odowska-Curie grant agreement No.\ 945332. \includegraphics*[scale = 0.028]{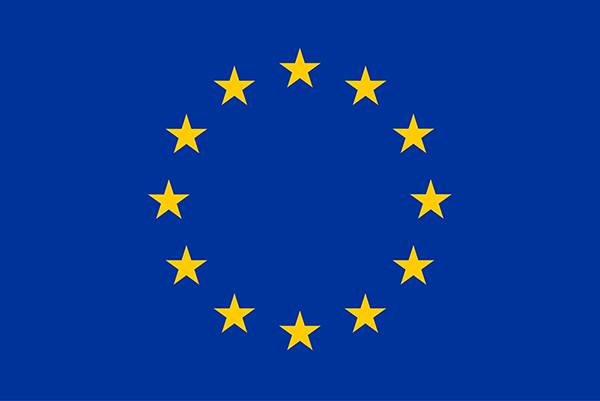}


%% file: bitoro.pdf_tex
\begingroup%
  \makeatletter%
  \providecommand\color[2][]{%
    \errmessage{(Inkscape) Color is used for the text in Inkscape, but the package 'color.sty' is not loaded}%
    \renewcommand\color[2][]{}%
  }%
  \providecommand\transparent[1]{%
    \errmessage{(Inkscape) Transparency is used (non-zero) for the text in Inkscape, but the package 'transparent.sty' is not loaded}%
    \renewcommand\transparent[1]{}%
  }%
  \providecommand\rotatebox[2]{#2}%
  \newcommand*\fsize{\dimexpr\f@size pt\relax}%
  \newcommand*\lineheight[1]{\fontsize{\fsize}{#1\fsize}\selectfont}%
  \ifx\svgwidth\undefined%
    \setlength{\unitlength}{329.77243739bp}%
    \ifx\svgscale\undefined%
      \relax%
    \else%
      \setlength{\unitlength}{\unitlength * \real{\svgscale}}%
    \fi%
  \else%
    \setlength{\unitlength}{\svgwidth}%
  \fi%
  \global\let\svgwidth\undefined%
  \global\let\svgscale\undefined%
  \makeatother%
  \begin{picture}(1,0.39403007)%
    \lineheight{1}%
    \setlength\tabcolsep{0pt}%
    \put(0,0){\includegraphics[width=\unitlength,page=1]{bitoro.pdf}}%
  \end{picture}%
\endgroup%

%% file: construction_radalla.tex
\section{Thickening of a geodesic}

Fix a closed hyperbolic surface $S$. Given a complete geodesic $\alpha$ in $S$, we denote by $\dot{\alpha}$ its lift to the unit tangent bundle of $S$.
The main goal of this section is to prove Proposition \ref{prop:simple-ends}.

We proceed as follows. First, we identify and prove fundamental properties of geodesics that are asymptotically simple. Next, we study the structure of the closure of a geodesic in $\boundangle$, both as a subset of $S$ and of $T^1S$. Finally, we construct a thickening of a geodesic (see Definition~\ref{def:thickening}), which we then use to lift $\alpha$ to an asymptotically simple geodesic.

\subsection{Asymptotically simple geodesics.}

We will study the set of geodesics that do not intersect themselves from some point onwards (or backwards). 

\begin{defi}\label{def:ends-simply}
 Let $\alpha$ be a geodesic. We say that $\alpha$ \textit{is asymptotically simple in the future} (resp. \textit{in the past}) if there exists a constant $K \in \RR$ such that the restriction $\alpha|_{[K, +\infty)}$ (resp. $\alpha|_{(-\infty, K]}$) is a simple curve. 
 If both conditions are true, we say that the geodesic is \textit{asymptotically simple}. We denote the set of such geodesics by $\simpends$.
\end{defi}

These notions can be restated in dynamical terms. Let $\alpha: \RR \to S$ be a geodesic. We say that a point $(p, v) \in T^1S$ is an $\omega^+$-limit (resp. $\omega^-$-limit) point of $\alpha$ if there exists a sequence $t_n \in \RR$ such that $\lim_{n \to +\infty} t_n = +\infty$ (resp. $-\infty$) and
\[
\lim_{n \to +\infty} \dot{\alpha}(t_n) = (p, v).
\]
We denote the $\omega^+$-limit set (resp. $\omega^-$-limit set) of $\alpha$ by $\omega^+(\alpha)$ (resp. $\omega^-(\alpha)$). It is not hard to show that the $\omega$-limit sets are closed and geodesically invariant (i.e. if a point in $T^1S$ is in $\omega^+$ or $\omega^-$, then the whole geodesic is in the set).

\begin{rmk}
 What is often referred to as the $\alpha$-limit set in the dynamical systems' literature is denoted here by $\omega^-$ in order to avoid confusion with our notation for geodesics.
\end{rmk}

It is not hard to show that if a geodesic $\alpha$ is asymptotically simple in the future, then every geodesic in $\omega^+(\alpha)$ is simple; similarly, if $\alpha$ is asymptotically simple in the past, then every geodesic in $\omega^-(\alpha)$ is simple.

\begin{rmk}
    Let $\gamma \subseteq \omega^+(\alpha)$ and $\tilde{\alpha}$ a lift of $\alpha$ to $\HH^2$. It is not necessarily true that there is a lift of $\gamma$ that is asymptotic to $\tilde{\alpha}$. A counterexample is a dense geodesic, whose $\omega^+$-limit set is the whole unit tangent bundle.
\end{rmk}

The following lemma shows that if two geodesics intersect, then geodesics that are sufficiently close to them intersect as well, at a nearby point and with approximately the same angle. This property is important in understanding both asymptotically simple geodesics and geodesics in $\boundangle$. The reader can see Figure \ref{fig:continuidad-intersecciones} for an example.

\begin{figure}[h]
    \centering
    \def\svgwidth{0.5\columnwidth}
    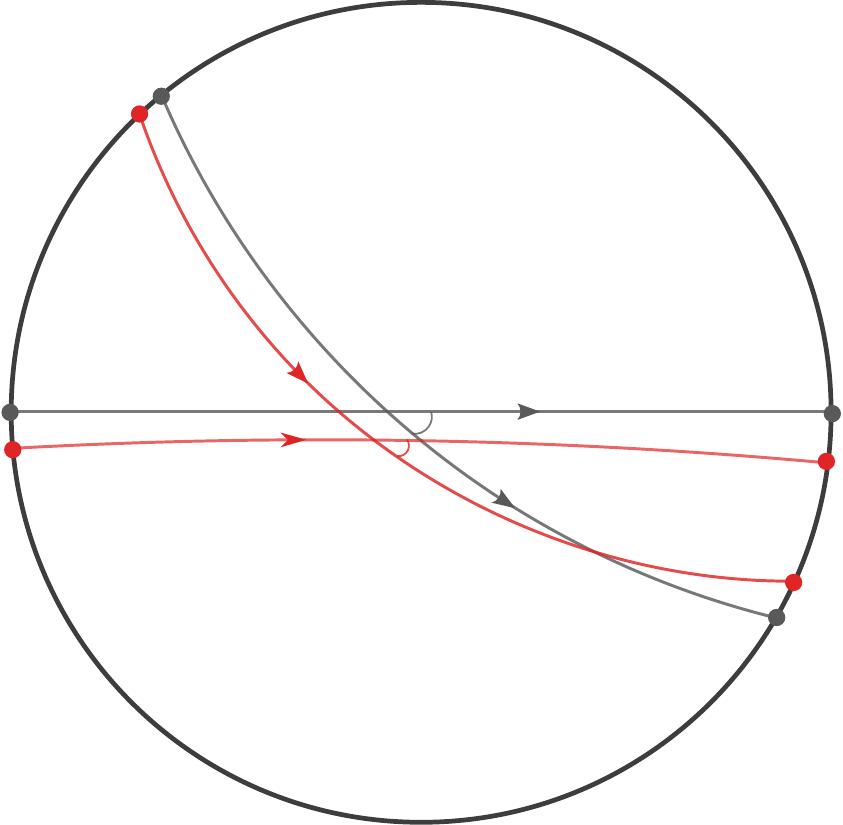
    \caption{After a small perturbation of the two pairs of gray points at infinity which bound geodesics that self intersect, we see the red geodesics, that have an intersection with similar angle close to the original intersection point.}
    \label{fig:continuidad-intersecciones}
  \end{figure}


The fact that the function which assigns to two intersecting geodesics in $\HH^2$ their angle is continuous gives rise to the following needed lemma.

\begin{lem}\label{lem:continuidad-intersecciones}
 Let $(p, v), (p, v') \in T^1S$. For every $\epsilon_1, \epsilon_2 > 0$ there exists an $\eta > 0$ such that if $(q_1, v_1) \in B_{T^1S}((p, v), \eta)$ and $(q_2, v_2) \in B_{T^1S}((p, v'), \eta)$, and $\alpha_1$ and $\alpha_2$ are the geodesics through $(q_1, v_1)$ and $(q_2, v_2)$ respectively then $\alpha_1$ and $\alpha_2$ intersect at a point $q \in B(p, \epsilon_1)$ and $$|\angle_p(\alpha_1, \alpha_2) - \angle(v, v')| < \epsilon_2.$$
 Moreover, $\eta$ only depends on the intersection angle. 
\end{lem}



The following lemma is a more relaxed condition than what we observed in the remark.

\begin{lem}\label{lem:a-la-estable}
Let $\alpha$ be a geodesic that is asymptotically simple in the future. Let $\tilde{\alpha}$ be a lift of $\alpha$ to $\HH^2$.
Then there exists a simple geodesic $\beta$ and a lift $\tilde{\beta}$ of $\beta$ such that $\tilde{\beta}(+ \infty) = \tilde{\alpha}(+ \infty)$.
\end{lem}
The strategy of the proof is to approximate $\alpha$ by long closed geodesics that fellow travel $\alpha$. Then, we take the limit and prove that the limiting geodesic (which is not necessarily closed) is simple, using that $\alpha$ itself is simple.
\begin{proof}

 Take $(p, v) \in \omega^+(\alpha)$. By definition there exists a sequence $t_n \to +\infty$ such that 
    $$\lim_{n \to \infty}\dot{\alpha}(t_n) = (p, v).$$
 Using the Anosov closing lemma for flows (see the book of Fisher-Hasselblatt \cite{fisher2019hyperbolic}, Theorem 5.3.11), for each $k \in \NN$ there exists $\epsilon_k > 0$ such that if, for $t, s \in \RR$, we have $d(\dot{\alpha}(t), \dot{\alpha}(s)) < \epsilon_k$ then there is a closed geodesic $\gamma$ with $\dot{\gamma} \subseteq \mathcal{N}_{T^1S}(\dot{\alpha}|_{[t, s]}, 1/k)$, the tubular neighborhood of the geodesic in $T^1S$.

 Now, because $\dot{\alpha}(t_n)$ is a Cauchy sequence in $T^1S$, there exists $n_k$ such that for every $n > n_k$ we have that $$d(\dot{\alpha}(t_n), \dot{\alpha}(t_{n_k})) < \epsilon_k.$$

 Thus, for every $n > n_k$, we define $\gamma_{n, k}$ as the closed geodesic that arises from applying the closing lemma to $\alpha|_{[t_{n_k}, t_n]}$. After passing to a subsequence, we may assume that the sequence $\gamma_{n, k}$ converges (in the set of parametrized geodesics of $S$ with the compact open topology) with $n$ to a geodesic which we call $\gamma_{\infty, k}$. 

 We consider $\gamma_{\infty, \infty}$ the limit in $k$ of the sequence $\gamma_{\infty, k}$. Now, let $t \in \RR$ such that $t \in [t_{n_k}, t_n]$ for some $n > n_k$. Then
    $$d(\alpha(t), \gamma_{\infty, k}) \le d(\alpha(t), \gamma_{n, k}) + d(\gamma_{n, k}, \gamma_{\infty, k}).$$

 Taking the limit of the inequality as $t$ goes to $+ \infty$ (and thus considering bigger $n$) we get that
    $$\limsup_{t\to +\infty} d(\alpha(t), \gamma_{\infty, k}) \le 1/k.$$
 Now, let $k(t)$ be the biggest $k$ such that $t_{n_k} < t$. We have that
    $$d(\alpha(t), \gamma_{\infty, \infty}) \le d(\alpha(t), \gamma_{\infty, k(t)}) + d(\gamma_{\infty, k(t)}, \gamma_{\infty, \infty}).$$

 Because $\lim_{t \to +\infty}k(t) = +\infty$ we have
    $$\lim_{t \to \infty} d(\alpha(t), \gamma_{\infty, \infty}) = 0,$$
and hence a lift of $\gamma_{\infty,\infty}$ is asymptotic to $\tilde{\alpha}$.

 In order to conclude the lemma, we need to prove that $\gamma_{\infty, \infty}$ is simple. Assume by contradiction that there exist $t_0, s_0$ such that $\gamma_{\infty, \infty}(t_0) = \gamma_{\infty, \infty}(s_0).$ There exists $\epsilon > 0$ such that, if $\beta$ is a geodesic and $t, s \in \RR$ are such that $d(\dot{\gamma}_{\infty, \infty}(t_0), \dot{\beta}(t)) < \epsilon$ and $d(\dot{\gamma}_{\infty, \infty}(s_0), \dot{\beta}(s)) < \epsilon$, then $\beta$ will have a self intersection.

 Let $k \in \NN$ big enough so $1/k < \epsilon/3$, and so that there exist $t_1$ and $s_1 \in \RR$ with
 $$d(\gamma_{\infty, \infty}(t_0), \gamma_{k, \infty}(t_1)) < \epsilon/3$$ and
  $$d(\gamma_{\infty, \infty}(s_0), \gamma_{k, \infty}(s_1)) < \epsilon/3.$$
  Again, because of the definition of $\gamma_{k, n}$ and $\gamma_{k, \infty}$, for $n$ big enough, we have that for some $t_2, s_2 \in \RR$ we get
   $$d(\gamma_{k, \infty}(t_1), \gamma_{k, n}(t_2)) < \epsilon/3$$ and
  $$d(\gamma_{k, \infty}(s_1), \gamma_{k, n}(s_2)) < \epsilon/3.$$
  Because of the definition of $\gamma_{k, n}$ and unraveling the three inequalities we get that there exist $t_3, s_3 \in \RR$ such that
  $$d(\gamma_{\infty, \infty}(t_0), \alpha(t_3)) < \epsilon$$ and
  $$d(\gamma_{\infty, \infty}(s_0), \alpha(s_3)) < \epsilon,$$
which is a contradiction since $\alpha$ is simple, and thus concludes the proof.
\end{proof}

In other words, the previous lemma implies that, if $\pi: \partial \HH^2 \times \partial \HH^2 \to \partial \HH^2$ denotes the projection onto the first coordinate, and $\mathscr{G}_s \subseteq \mathscr{G}$ denotes the set of simple geodesics, Lemma~\ref{lem:a-la-estable} shows that the set of asymptotically simple geodesics is a subset of
$$\pi(\mathscr{G}_s) \times \pi(\mathscr{G}_s).$$
This observation will be key in proving that the set $\simpends$ has Hausdorff dimension $0$. 

The statement of the following proposition relates the set $\boundangle$ with asymptotically simple geodesics.

\begin{prop}\label{prop:simple-ends}
 Let $\alpha \in \boundangle$. Then there exists a finite cover $p: \tilde{S} \to S$ and a curve $\tilde{\alpha}$ in $\tilde{S}$ such that $p(\tilde{\alpha}) = \alpha$ and $\tilde{\alpha} \in \simpends$.
\end{prop}

To prove this, we begin by studying the set $\overline{\alpha}$, the closure of $\alpha$ as a subset of $S$. This is a geodesically closed set, that is, a union of complete geodesics. In fact, 
$$\overline{\alpha} = \omega^-(\alpha) \cup \alpha \cup \omega^+(\alpha).$$
 When $\alpha \in \boundangle$, the set $\overline{\alpha}$ exhibits structural similarities to geodesic laminations. For an introduction to geodesic laminations, see Harer and Penner's book \cite{penner1992combinatorics} or Thurston's notes \cite{thurston2022geometry}.

\subsection{The closure of $\alpha$.}\label{sec:clausura}

Given a geodesic $\alpha$, we will study its closure in $S$ (denoted by $\overline{\alpha}$) and in the unit tangent bundle $T^1S$ (denoted by $\overline{\dot{\alpha}}$). Note that if $(p, v) \in \overline{\dot{\alpha}}$, then the entire geodesic through $(p, v)$ also lies in $\overline{\dot{\alpha}}$.

As a consequence of Lemma~\ref{lem:continuidad-intersecciones}, if $\alpha \in \boundangle$ for some $\delta > 0$, geodesics in $\overline{\alpha}$ cannot intersect other geodesics in $\overline{\alpha}$ at angles smaller than $\delta$. We formalize this in the following lemma.

\begin{lem}\label{lem:estructura-clausura}
 Let $\alpha \in \boundangle$. Let $p \in S$ be such that there exist $t_1, t_2 \in \RR$ and two geodesics $\alpha_1, \alpha_2 \in \overline{\alpha}$ such that $\alpha_1(t_1) = \alpha_2(t_2) = p$. Then
    $$\angle_p (\alpha_1, \alpha_2) \ge \delta.$$
\end{lem}

\begin{proof}
 Let $\alpha_1, \alpha_2$ be geodesics in $\overline{\alpha}$, and assume there exist $t_1, t_2$ such that $\alpha_1(t_1) = \alpha_2(t_2)$ and $a = \angle (\dot{\alpha_1}(t_1), \dot{\alpha_2}(t_2)) < \delta.$ For every $\epsilon > 0$ there exist $s_1, s_2 \in \RR$ such that 
    $$d\left(\alpha_1(t_1), \alpha(s_1)\right) < \epsilon $$ and $$d\left(\alpha_2(t_2), \alpha(s_2)\right) < \epsilon.$$
 Because of Lemma \ref{lem:continuidad-intersecciones}, this means that we can find self-intersection points in $\alpha$ with an angle arbitrarily close to $a$, which is a contradiction because $\alpha \in \boundangle$.
\end{proof}

A consequence of Lemma \ref{lem:estructura-clausura} is the following result.

\begin{cor}\label{cor:finite-Ap}
 There exists $N \in \NN$, only depending on $\delta$, such that if $\alpha \in \boundangle$ and $p \in \overline{\alpha}$, then the number of tangent vectors of geodesics in $\overline{\alpha}$ through $p$ is bounded by $N$.
\end{cor}

\begin{proof}
 Because the angles of vectors in $\overline{\dot{\alpha}}$ at a point $p$ need to be $\delta$-separated, it is enough to take $N = \lfloor 2\pi/\delta\rfloor + 1$.
\end{proof}

Due to the ergodicity of the geodesic flow and Lemma~\ref{lem:continuidad-intersecciones}, we also get the following corollary.

\begin{cor}\label{cor:empty-interior}
    $\overline{\dot{\alpha}}$ and $\overline{\alpha}$ have empty interior in $T^1S$ and $S$ respectively.
\end{cor}

\begin{proof}
 The empty interior on $T^1S$ follows from the ergodicity of the geodesic flow (see the book of Fisher-Hasselblatt \cite{fisher2019hyperbolic}).

 In order to prove that $\overline{\alpha}$ has empty interior on $S$, assume by contradiction that there exists an open set $V \subseteq S$ contained in $\overline{\alpha}$. Take a geodesic segment $L$ in $V$. By Corollary~\ref{cor:finite-Ap}, for every $p \in \overline{\alpha}$ the set 
    $$\{ v \in T^1_pS: (p, v) \in \overline{\dot{\alpha}}\}$$ is finite, which means that there exists a subsegment $L' \subseteq L$ and a continuous section $s: L' \to \overline{\dot{\alpha}}$.

 Consider a lift $\tilde{L}'$ to $\HH^2$ (see Figure \ref{fig:interior-vacio}), and a section $s: \tilde{L}' \to T^1\HH^2$ that factors the section in $S$. Consider also, for a point $(p, v) \in T^1\HH^2$, the geodesic $\alpha_{p, v}$ through $(p,v)$. Then the maps
    $$
    \begin{aligned}
 f_s: \tilde{L'} &\to \partial \HH^2\\
 (p, v) &\mapsto \alpha_{p,s(p)}(+\infty)
    \end{aligned}
    $$
 and
    $$
    \begin{aligned}
 f_u: \tilde{L'} &\to \partial \HH^2\\
 (p, v) &\mapsto \alpha_{p,s(p)}(-\infty)
    \end{aligned}
    $$
are well-defined and continuous. Because the image of a connected set is connected, either $f_s(\tilde{L}')$ or $f_u(\tilde{L}')$ contains an open set of $\partial\HH^2$.
 Assume, without loss of generality, that $f_s(L')$ contains an open set. By the ergodicity of the geodesic flow, $L'$ must contain a forward dense geodesic, but this is a contradiction because $L' \subseteq \overline{\alpha}$.
\end{proof} 

\begin{figure}[h]
    \centering
    \def\svgwidth{0.7\columnwidth}
    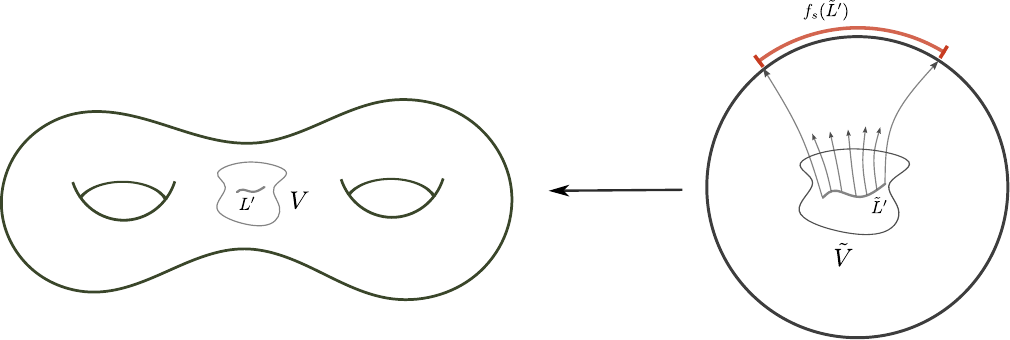
    \caption{If we take a segment $\tilde{L}'$ in $\HH^2$ (right), and a section of the segment, then either $f_u(\tilde{L'})$ or $f_s(\tilde{L'})$ will contain an interval on the boundary. Thus, one of the geodesics whose tangent belongs to the image of the section must be dense. See Corollary \ref{cor:empty-interior}.}
    \label{fig:interior-vacio}
  \end{figure}

\subsection{Band complex and thickenings of geodesics.}

Consider a geodesic in $\boundangle$ for some $\delta > 0$. By Corollary \ref{cor:empty-interior} we know that its closure has empty interior. We will study a subsurface that encompasses the information of self-intersections in the closure: a band complex. In the literature, band complexes are used to build a bridge between train tracks and simple geodesics, however in Erlandsson-Souto's paper \cite{erlandsson2016counting} this idea is generalized to give information about geodesics with self-intersections. In this paper we slightly adapt the definitions of the book to fit our interests.

\begin{defi}
    A \textit{band} is a rectangle $[0, a] \times [0, b]$, foliated by the vertical segments $\{t\} \times [0, b]$. A \textit{band complex} is a space $X$ obtained by gluing finitely many bands together along the vertical sides in such a way that the boundary of $X$ consists of the horizontal boundaries of the bands. Observe that this yields a well-defined foliation $\mathcal{F}_X$ on $X$. The foliation is called \textit{vertical foliation.}

    If $X$ is obtained by gluing finitely many sides and there are vertical boundary components we say that $X$ is an \textit{incomplete} band complex.
\end{defi}

\begin{figure}[h]
    \centering
    \def\svgwidth{0.4\columnwidth}
    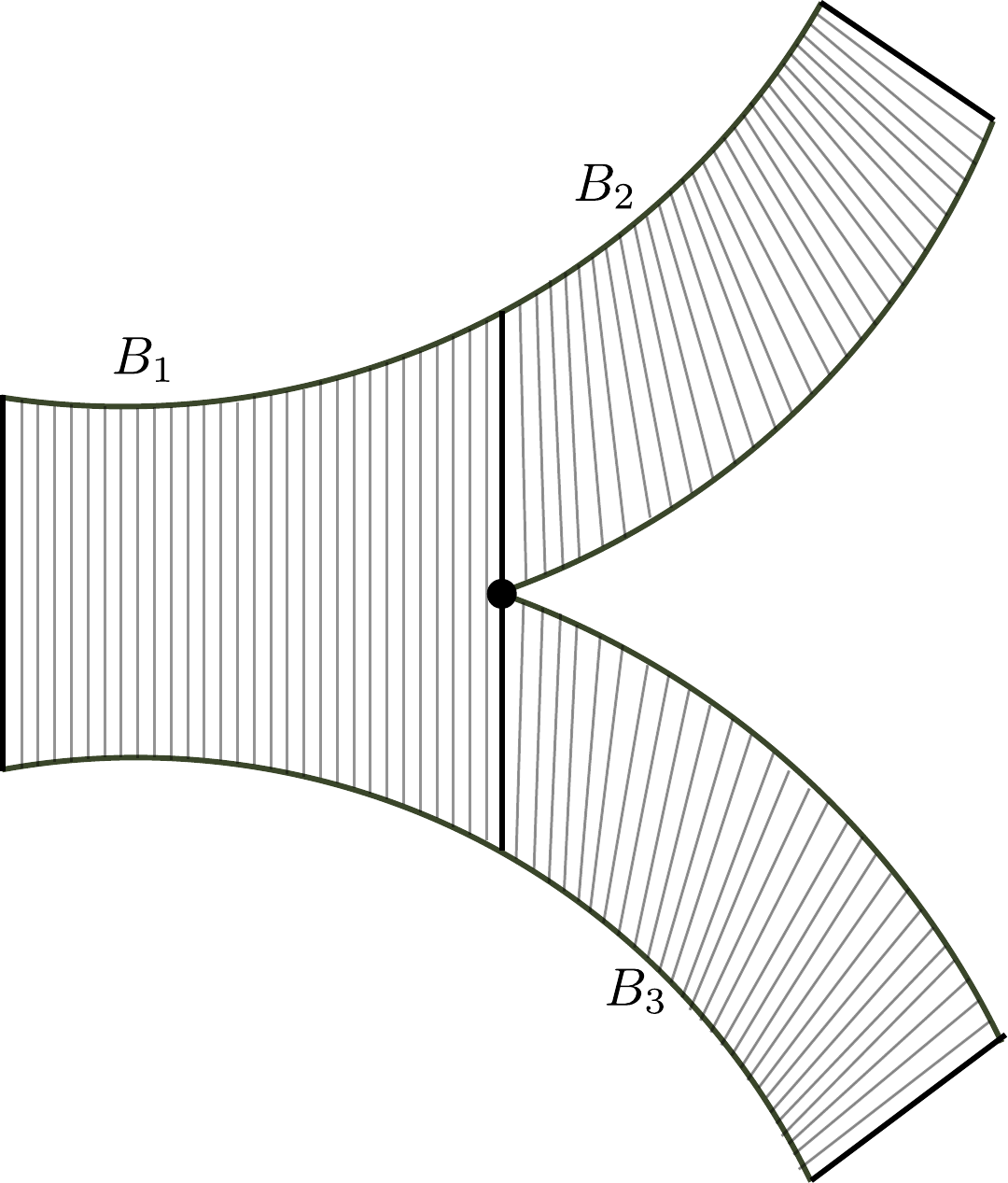
    \caption{An example of an incomplete band complex with their respective vertical foliations. Observe that several bands can be glued in the same vertical side. }
    \label{fig:banda}
  \end{figure}

Observe that a band complex is a compact space. In Figure \ref{fig:banda} we can see an example of an incomplete band complex. Also inspired by the same book, we adapt the definition of the \textit{thickening of a radalla} to give a definition of the \textit{thickening of (sets of) geodesics.}

\begin{defi}\label{def:thickening}
    Let $\alpha$ be a geodesic (not necessarily closed) or a set of geodesics. A \textit{thickening of $\alpha$} is a band complex $X$, together with an immersion $i: X \to S$ and a curve $\tilde{\alpha}$ in $X$ such that:
    \begin{enumerate}
        \item $i(\tilde{\alpha}) = \alpha$
        \item $\tilde{\alpha}$ is transversal to the vertical foliation.
        \item $i(X) \subseteq B(\epsilon, \alpha)$.
    \end{enumerate}
    If $(i, X)$ is a thickening of $\alpha$ we say that $\alpha$ is \emph{carried by} $(i, X)$.
\end{defi}

In Figure \ref{fig:immersed}, we see a possible intersection between two bands in a band complex. Intersections are not necessarily square-shaped, they can be more complicated. For a depiction of a more complicated intersection, see Figure \ref{fig:thickening}.

\begin{figure}[h]
    \centering
    \def\svgwidth{0.8\columnwidth}
    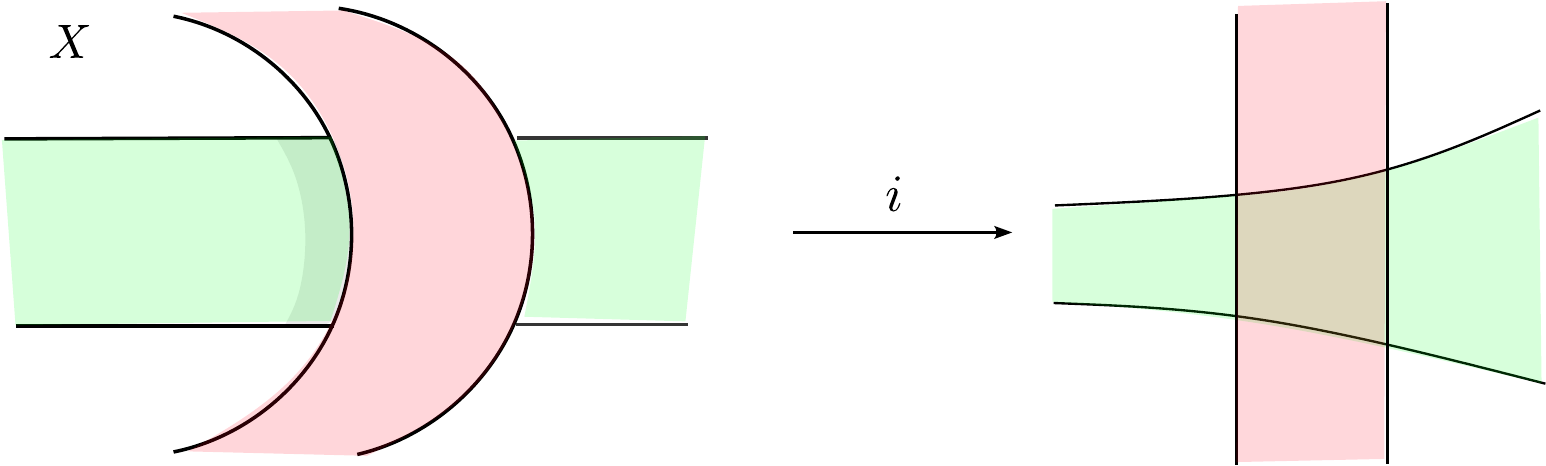
    \caption{Two bands in the band complex $X$ such that their image under the immersion $i$ intersects.}
    \label{fig:immersed}
  \end{figure}

Instead of constructing a thickening of $\alpha$, by using the results of Section \ref{sec:clausura} we will construct a thickening of $\overline{\alpha}$. The following remark will be useful to lift thickenings to covers of $S$.

\begin{rmk} \label{rmk:cover-of-thickening}
    For $\epsilon$ smaller than half the systole, the following holds: if $(i, X)$ is a thickening of a geodesic $\alpha \subset S$, with image contained in the $\epsilon$-neighborhood of $\alpha$, and $p: S' \to S$ is a finite covering map such that $\alpha' \subset S'$ is a geodesic satisfying $p(\alpha') = \alpha$, then there exists a thickening $(i', X)$ of $\alpha'$ such that $p \circ i' = i$.
\end{rmk}

\begin{proof}
    For such an $\epsilon$, if $x \in S$, then $p^{-1}(B(x, \epsilon))$ is a union of disjoint balls around lifts of $x$. Thus, the $\epsilon$ neighborhood of $\alpha$ lifts to an $\epsilon$-neighborhood of $\alpha'$, and thus we can lift the map $i$ to $S'$.
\end{proof}

The main tool of this section is the following proposition, that finds a thickening that carries the intersection information of a geodesic $\alpha \in \boundangle$.

\begin{prop}\label{prop:band-complex}
    Let $\alpha \in \boundangle$. Then for sufficiently small $\epsilon > 0$, the tubular neighborhood $\mathcal{N}(\alpha, \epsilon)$ can be given the structure of a band complex $(i, X)$, and is thus a thickening of $\alpha$. Moreover, the following properties hold:

    \begin{enumerate}
        \item Intersections of $\overline{\alpha}$ cannot happen within the same band.
        \item If two geodesic arcs $\gamma_1, \gamma_2 \subseteq \overline{\alpha}$ intersect in bands $B_1, B_2$, then every geodesic carried by $X$ that intersects $i(B_1)$ and $i(B_2)$ has a self-intersection.
    \end{enumerate}
\end{prop}

In order to construct a band structure for the tubular neighborhood of $\overline{\alpha}$, we need to identify the directions along which to build the bands. Assuming the geodesic lies in $\boundangle$, we can detect ``parallel'' directions by locating arcs that are nearly perpendicular to segments of the geodesic. The following notion of approximate perpendicularity will be useful in describing the geometry of the tubular neighborhood and constructing the bands.

\begin{defi}
    Let $\gamma_1, \gamma_2$ be two embedded arcs in $S$ such that there exists $t \in \RR$ with $\gamma_1(t) = \gamma_2(t)$. We say that $\gamma_1$ and $\gamma_2$ are \textit{$\eta$-perpendicular} if
    $$|\angle(\dot{\gamma}_1(t), \dot{\gamma}_2(t)) - \pi/2| < \eta.$$
\end{defi}

The following lemma shows that, locally, one can identify a finite collection of directions that geodesics in $\boundangle$ tend to follow.

\begin{lem}\label{lem:las-transversales}
    Let $\alpha \in \boundangle$. For $\eta > 0$ small enough (only depending on $\delta$) and for every point $p \in \overline{\alpha}$ there exists a neighborhood $U$ of $p$ and a finite collection of geodesic arcs $\mathcal{I}$ such that the following holds: for every arc $\gamma$ in $\overline{\alpha} \cap U$, there exists a unique $I \in \mathcal{I}$ such that $\gamma$ and $I$ are $\eta$-perpendicular.
\end{lem}

\begin{proof}
    Let $p \in S$ and consider $\eta < \delta/10$. Let $$A_p = \{v \in T^1_pS: (p, v) \in \overline{\dot{\alpha}}\}.$$
    Because of Corollary \ref{cor:finite-Ap} it is finite set. Take $V = B(p, \epsilon)$, for  $\epsilon$ small. For each $(p, v)$ in $A_p$, let $I_v$ be the geodesic diameter of $V$ whose tangent at $p$ is perpendicular to $(p, v)$. By Lemma \ref{lem:continuidad-intersecciones}, there exists $\epsilon_2$ small enough such that every geodesic arc $\gamma \in \overline{\alpha} \cap B(p, \epsilon_2)$ will be $\eta$-perpendicular to one of the $I_v$. Because $\eta < \delta/10$, none of the arcs of $\overline{\alpha} \cap U$ can be $\epsilon$-perpendicular to two arcs $I_v$, $I_w$. The uniform bound comes from the compactness of $S$.
\end{proof}

\begin{rmk}\label{rmk:perturbacion-transversales}
    In the proof of Lemma \ref{lem:las-transversales}, the arcs in $\mathcal{I}$ were chosen to be perpendicular to each tangent vector $(p, v) \in A_p$. However, the construction is not unique: if we slightly rotate each arc around the point $p$, we obtain a different collection $\mathcal{I}$ that still satisfies the same $\eta$-perpendicularity condition.
\end{rmk}

By a continuity argument and Lemma \ref{lem:estructura-clausura} we also have the following lemma.

\begin{lem}\label{lema:paralelos-no-intersectan}
    For every $\eta > 0$ there exists $\epsilon > 0$ such that, whenever $\gamma, \gamma'$ are two geodesic arcs of length smaller than $\epsilon$, and both are $\eta$-perpendicular to a geodesic arc $I$, then 
    \begin{enumerate}
        \item\label{it:prop1} if there exists $p \in S$ and $t, t'$ such that $p = \gamma(t) = \gamma(t')$ then we have
    $$\angle_p(\gamma, \gamma') < 3\eta$$
        \item\label{it:prop2} If there is another arc $I'$ that is $\eta$-perpendicular to $\gamma$ then $I'$ is $3\eta$-perpendicular to $\gamma'$.
    \end{enumerate} 
\end{lem}

The previous lemma could be reformulated in the following way: if a geodesic is in $\boundangle$, then two segments cannot intersect if they are almost parallel.
Now we are ready to prove Proposition \ref{prop:band-complex}

\begin{proof}[Proof of Proposition \ref{prop:band-complex}.]
    Fix $\eta < \delta/6$. By Lemma \ref{lema:paralelos-no-intersectan} there exists an $\epsilon_2 > 0$ such that, if two arcs $\gamma, \gamma' \subseteq \overline{\alpha}$ of length smaller than $\epsilon_2$ are $\eta-$perpendicular to the same arc then they cannot intersect.
       
    For each $p \in \overline{\alpha}$ we take a neighborhood $U_p$ as in Lemma \ref{lem:las-transversales}, with their respective finite collection of arcs $\mathcal{I}_p$. Because $\overline{\alpha}$ is compact we can take a finite number of neighborhoods $U_{p_1}, \ldots, U_{p_m}$ that cover $\overline{\alpha}$ with their respective collection of arcs 
       $$\mathcal{I} = \bigcup_{i = 1}^m \mathcal{I}_{p_i}.$$
       Since the collection $\{U_{p_i}\}$ covers $\overline{\alpha}$, and by the construction of the intervals, every geodesic $\alpha' \subset \overline{\alpha}$ can be written as a union of geodesic segments $\gamma: [a, b] \to S$, where there exist intervals $I, I' \in \mathcal{I}$ such that $\gamma$ is $\eta$-perpendicular to $I$ at $\gamma(a)$ and to $I'$ at $\gamma(b)$. 

       Moreover, by Remark~\ref{rmk:perturbacion-transversales}, and after a slight perturbation of the intervals if necessary, we may assume that no two intervals in $\mathcal{I}$ intersect at a point in $\overline{\alpha}$. Additionally, after possibly shortening the intervals, we may further assume that $\partial I \cap \overline{\alpha} = \emptyset$ for every $I \in \mathcal{I}$.

   
    Now we construct a graph $(V, E)$ such that $V = \mathcal{I}$ is the set of vertex, and $(I_i, I_j) \in E$ if there exists a geodesic arc in $\overline{\alpha}$ that is $\epsilon$-perpendicular to $I_i$ and $I_j$ and is not $\eta$-perpendicular to any other element of $\mathcal{I}$. We denote by $\alpha_{ij}$ the set of such arcs.
   
    For each $I_i \in V$, consider 
    $$\alpha_i = \{\cup \alpha_{ij}: (I_i, I_j)  \in E\}.$$
    By construction, whenever a geodesic arc crosses a ball $U_i$, we will find an arc $I$ that is $\eta$-perpendicular to the geodesic. Thus, every arc in $\alpha_i$ will have a length smaller than $6 \epsilon$, and because $\overline{\alpha}$ does not cross any intersection of arcs, the length is uniformly bounded below. This implies that $\alpha_{i}$ is a closed, simple subset of geodesic arcs. 
    Consider $\epsilon^i_0 > 0$ such that the tubular neighborhood $\mathcal{N}(\alpha_i, \epsilon^i_0)$ does not intersect the boundary of any line in $\mathcal{I}$ or any of the intersection points of arcs of $\mathcal{I}$. 
    
    We will adapt the argument in Erlandsson-Souto's book, page 62 (\cite{erlandsson2016counting}) to construct a thickening of $\alpha_i$. If we take $\epsilon^i_0$ small enough we can assume that the boundary of $\mathcal{N}(\alpha_i, \epsilon^i_0)$ is a piecewise $C^1$ curve, with a finite number of singularities. Thus, we can divide $\mathcal{N}(\alpha_i, \epsilon^i_0)$ into a finite union of bands $B_i$, where the vertical boundaries of the bands are either subsegments of $I \in \mathcal{I}_i$ or geodesics transversal to $\partial \mathcal{N}(\alpha_i, \epsilon^i_0)$, going through the singularities. It is not hard to see that we can then foliate each $B_i$ to build a vertical foliation transversal to the arcs in $\alpha_i$. See Figure \ref{fig:thickening} to see two bands that intersect each other.
   
    By taking 
    $$\epsilon_2 = \min_{i \in m} \{ \epsilon^i_0, \epsilon \}$$
     we can glue all the incomplete band complexes together to build a thickening of $\overline{\alpha}$. Again, because of Lemma \ref{lema:paralelos-no-intersectan}, we have that no two arcs in $\alpha_i$ can intersect, and thus, two geodesics in the same band cannot intersect, which proves the property (\ref{it:prop1}) of the proposition.
     

    In order to prove property (\ref{it:prop2}) of the proposition, first observe that for every $\epsilon < \epsilon_2$, the previous construction yields a band complex structure on the tubular neighborhood $\mathcal{N}(\overline{\alpha}, \epsilon)$. Let $(X_\epsilon, i_\epsilon)$ the band complex structure in such tubular neighborhood. 
    Because 
     $$\bigcap_{\epsilon > 0} \mathcal{N}(\overline{\alpha}, \epsilon) = \overline{\alpha},$$
    for every $\tau > 0$ there exists $\epsilon_\tau > 0$ such that the following holds: if $\beta$ and $\gamma$ are two geodesics that pass through a band $B$ of $X_{\epsilon_\tau}$, and $t \in \RR$ is such that $\beta(t) \in B$, then $d(\dot{\beta}(t), \dot{\gamma} < \tau)$.
    Then property (2) follows after taking $\tau$ smaller than the constant of Lemma \ref{lem:continuidad-intersecciones}.    
   \end{proof}
   
   \begin{figure}[h]
    \centering
    \def\svgwidth{0.5\columnwidth}
    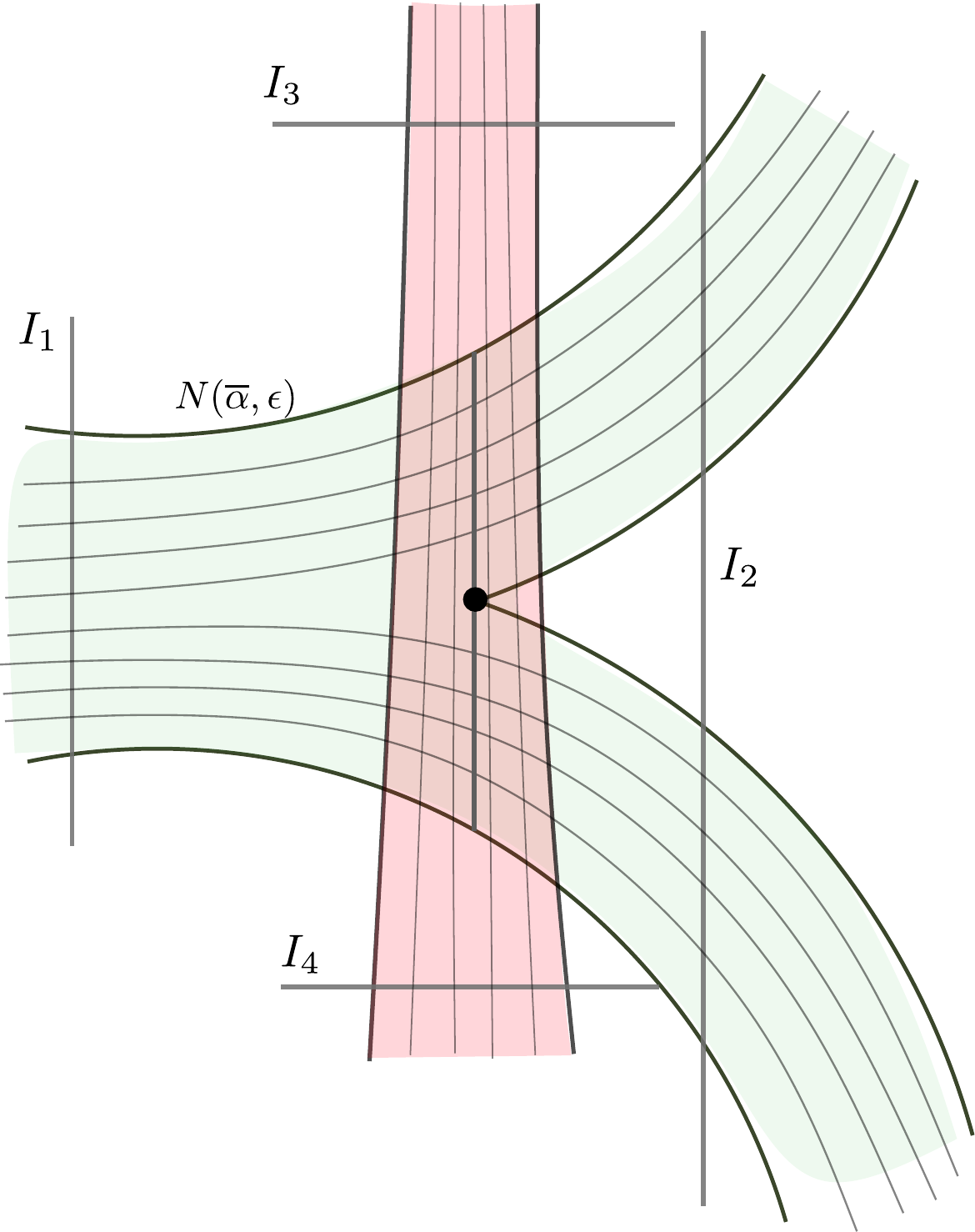
    \caption{A figure representing the construction of the bands for $\alpha_{1,2}$ and $\alpha_{3,4}$. Observe that any geodesic carried by these bands needs to have a self-intersection.}
    \label{fig:thickening}
  \end{figure}
   
   \subsection{Proof of Proposition \ref{prop:simple-ends}}

   Using the band complex we constructed in the previous subsection, we will prove Proposition \ref{prop:simple-ends}. We will identify which bands carry the self-intersections of the geodesic, and show we can remove these intersections through a finite cover.
   
   \begin{defi}
    Let $\alpha$ be a geodesic, and $(i, X)$ a thickening of $\alpha$ with $\alpha'$ the lift of $\alpha$ to $X$. We say that the band $B \subseteq X$ is \textit{recurrent} if, for every $K > 0$, there exists $t > K$ such that $\alpha'(t) \in B$. We say that a recurrent band is \textit{singularly recurrent} if there exists a recurrent band $B'$ such that $i(B) \cap i(B') \neq \emptyset$.
   \end{defi}

   \begin{rmk}\label{rmk:bandas-no-intersectan}
    Let $\alpha \in \boundangle$, and $(i, X)$ the band complex structure $\alpha$ given by Proposition \ref{prop:band-complex}. Then $\alpha$ is asymptotically simple if and only if there are no singularly recurrent bands.
   \end{rmk}
   
   \begin{proof}
    This follows directly from properties (1) and (2) of Proposition \ref{prop:band-complex}.
   \end{proof}

   Now, we are ready to prove the proposition.
   \begin{proof}[Proof of proposition \ref{prop:simple-ends}]
    We will show the fact that if $\alpha \in \boundangle$ then there exists a finite cover $p: S \to S'$ and a geodesic $\alpha'$ such that $\alpha'$ is asymptotically simple in the future and $p(\alpha') = \alpha$. Because an analogous proof works for the geodesics that are asymptotically simple in the past, proving this is enough to prove the proposition.
   
    Let $\epsilon_2$ be the constant given by Remark \ref{rmk:cover-of-thickening}, and let $(i, X)$ be the band complex structure on the $\epsilon$-tubular neighborhood of $\overline{\alpha}$ given by Proposition \ref{prop:band-complex} for $\epsilon < \epsilon_2$. Let $\tilde{\alpha}$ be the lift of $\alpha$ in $X$, and let $B$ be a recurrent band. Pick $L$ a leaf of the vertical foliation such that $i(L) \subseteq B$. Observe that there exists a sequence $\{t_n\}_{n \in \NN}$ such that $\alpha|_{[t_n, t_{n+1}]}$ intersects every singularly recurrent band and $\tilde{\alpha}(t_n) \in L$ for every $i \in \NN$. Because the leaves of the vertical foliation are closed, we can assume that $\alpha(t_n)$ converges to a point $p \in L$. Moreover, because $\tilde{\alpha}$ is simple, $\alpha(t_n)$ converges to some $i(p)$ also in $T^1S$.
   
    Because of the Anosov closing lemma for flows (see the book of Fisher-Hasselblatt \cite{fisher2019hyperbolic}), there exists $n \in \NN$ and a closed geodesic $\gamma$ such that $\gamma \subseteq \mathcal{N}(\alpha_{[t_n, t_{n + 1}]}, \epsilon)$. This implies that each singularly recurrent band $B_i$ is contained in $\mathcal{N}(\gamma, 2\epsilon)$.
    
    Theorem 3.3 of Scott's paper \cite{scott1978subgroups} states that every closed geodesic can be lifted to a simple geodesic through a finite cover. As a direct application of this, there exists a finite cover $p: S' \to S$ and a simple closed curve $\gamma'$ such that $p(\gamma') = \gamma$. Choose a lift $\alpha'$ of $\alpha$ such that $\gamma' \subseteq \mathcal{N}(\alpha', \epsilon)$. Observe that $p$ factors the band complex $(i, X)$ to a band complex $(i', X)$ which carries $\alpha'$. Moreover, the set of recurrent bands for $i'$ is exactly the lifts to $S'$ of the recurrent bands for $i$. The fact that $\gamma'$ is simple implies that none of these bands can be singularly recurrent, which proves the proposition.   
   
   \end{proof}

%% file: continuidad-interseccion.pdf_tex
\begingroup%
  \makeatletter%
  \providecommand\color[2][]{%
    \errmessage{(Inkscape) Color is used for the text in Inkscape, but the package 'color.sty' is not loaded}%
    \renewcommand\color[2][]{}%
  }%
  \providecommand\transparent[1]{%
    \errmessage{(Inkscape) Transparency is used (non-zero) for the text in Inkscape, but the package 'transparent.sty' is not loaded}%
    \renewcommand\transparent[1]{}%
  }%
  \providecommand\rotatebox[2]{#2}%
  \newcommand*\fsize{\dimexpr\f@size pt\relax}%
  \newcommand*\lineheight[1]{\fontsize{\fsize}{#1\fsize}\selectfont}%
  \ifx\svgwidth\undefined%
    \setlength{\unitlength}{404.37416858bp}%
    \ifx\svgscale\undefined%
      \relax%
    \else%
      \setlength{\unitlength}{\unitlength * \real{\svgscale}}%
    \fi%
  \else%
    \setlength{\unitlength}{\svgwidth}%
  \fi%
  \global\let\svgwidth\undefined%
  \global\let\svgscale\undefined%
  \makeatother%
  \begin{picture}(1,0.97887921)%
    \lineheight{1}%
    \setlength\tabcolsep{0pt}%
    \put(0,0){\includegraphics[width=\unitlength,page=1]{continuidad-interseccion.pdf}}%
  \end{picture}%
\endgroup%

%% file: interior-vacio.pdf_tex
\begingroup%
  \makeatletter%
  \providecommand\color[2][]{%
    \errmessage{(Inkscape) Color is used for the text in Inkscape, but the package 'color.sty' is not loaded}%
    \renewcommand\color[2][]{}%
  }%
  \providecommand\transparent[1]{%
    \errmessage{(Inkscape) Transparency is used (non-zero) for the text in Inkscape, but the package 'transparent.sty' is not loaded}%
    \renewcommand\transparent[1]{}%
  }%
  \providecommand\rotatebox[2]{#2}%
  \newcommand*\fsize{\dimexpr\f@size pt\relax}%
  \newcommand*\lineheight[1]{\fontsize{\fsize}{#1\fsize}\selectfont}%
  \ifx\svgwidth\undefined%
    \setlength{\unitlength}{484.48374831bp}%
    \ifx\svgscale\undefined%
      \relax%
    \else%
      \setlength{\unitlength}{\unitlength * \real{\svgscale}}%
    \fi%
  \else%
    \setlength{\unitlength}{\svgwidth}%
  \fi%
  \global\let\svgwidth\undefined%
  \global\let\svgscale\undefined%
  \makeatother%
  \begin{picture}(1,0.33604364)%
    \lineheight{1}%
    \setlength\tabcolsep{0pt}%
    \put(0,0){\includegraphics[width=\unitlength,page=1]{interior-vacio.pdf}}%
  \end{picture}%
\endgroup%

%% file: banda.pdf_tex
\begingroup%
  \makeatletter%
  \providecommand\color[2][]{%
    \errmessage{(Inkscape) Color is used for the text in Inkscape, but the package 'color.sty' is not loaded}%
    \renewcommand\color[2][]{}%
  }%
  \providecommand\transparent[1]{%
    \errmessage{(Inkscape) Transparency is used (non-zero) for the text in Inkscape, but the package 'transparent.sty' is not loaded}%
    \renewcommand\transparent[1]{}%
  }%
  \providecommand\rotatebox[2]{#2}%
  \newcommand*\fsize{\dimexpr\f@size pt\relax}%
  \newcommand*\lineheight[1]{\fontsize{\fsize}{#1\fsize}\selectfont}%
  \ifx\svgwidth\undefined%
    \setlength{\unitlength}{517.34052523bp}%
    \ifx\svgscale\undefined%
      \relax%
    \else%
      \setlength{\unitlength}{\unitlength * \real{\svgscale}}%
    \fi%
  \else%
    \setlength{\unitlength}{\svgwidth}%
  \fi%
  \global\let\svgwidth\undefined%
  \global\let\svgscale\undefined%
  \makeatother%
  \begin{picture}(1,1.17450057)%
    \lineheight{1}%
    \setlength\tabcolsep{0pt}%
    \put(0,0){\includegraphics[width=\unitlength,page=1]{banda.pdf}}%
  \end{picture}%
\endgroup%

%% file: immersed-band.pdf_tex
\begingroup%
  \makeatletter%
  \providecommand\color[2][]{%
    \errmessage{(Inkscape) Color is used for the text in Inkscape, but the package 'color.sty' is not loaded}%
    \renewcommand\color[2][]{}%
  }%
  \providecommand\transparent[1]{%
    \errmessage{(Inkscape) Transparency is used (non-zero) for the text in Inkscape, but the package 'transparent.sty' is not loaded}%
    \renewcommand\transparent[1]{}%
  }%
  \providecommand\rotatebox[2]{#2}%
  \newcommand*\fsize{\dimexpr\f@size pt\relax}%
  \newcommand*\lineheight[1]{\fontsize{\fsize}{#1\fsize}\selectfont}%
  \ifx\svgwidth\undefined%
    \setlength{\unitlength}{741.94592574bp}%
    \ifx\svgscale\undefined%
      \relax%
    \else%
      \setlength{\unitlength}{\unitlength * \real{\svgscale}}%
    \fi%
  \else%
    \setlength{\unitlength}{\svgwidth}%
  \fi%
  \global\let\svgwidth\undefined%
  \global\let\svgscale\undefined%
  \makeatother%
  \begin{picture}(1,0.30097191)%
    \lineheight{1}%
    \setlength\tabcolsep{0pt}%
    \put(0,0){\includegraphics[width=\unitlength,page=1]{immersed-band.pdf}}%
  \end{picture}%
\endgroup%

%% file: thickening.pdf_tex
\begingroup%
  \makeatletter%
  \providecommand\color[2][]{%
    \errmessage{(Inkscape) Color is used for the text in Inkscape, but the package 'color.sty' is not loaded}%
    \renewcommand\color[2][]{}%
  }%
  \providecommand\transparent[1]{%
    \errmessage{(Inkscape) Transparency is used (non-zero) for the text in Inkscape, but the package 'transparent.sty' is not loaded}%
    \renewcommand\transparent[1]{}%
  }%
  \providecommand\rotatebox[2]{#2}%
  \newcommand*\fsize{\dimexpr\f@size pt\relax}%
  \newcommand*\lineheight[1]{\fontsize{\fsize}{#1\fsize}\selectfont}%
  \ifx\svgwidth\undefined%
    \setlength{\unitlength}{561.35214498bp}%
    \ifx\svgscale\undefined%
      \relax%
    \else%
      \setlength{\unitlength}{\unitlength * \real{\svgscale}}%
    \fi%
  \else%
    \setlength{\unitlength}{\svgwidth}%
  \fi%
  \global\let\svgwidth\undefined%
  \global\let\svgscale\undefined%
  \makeatother%
  \begin{picture}(1,1.2635678)%
    \lineheight{1}%
    \setlength\tabcolsep{0pt}%
    \put(0,0){\includegraphics[width=\unitlength,page=1]{thickening.pdf}}%
  \end{picture}%
\endgroup%

%% file: simple_ends.tex
\section{Hausdorff dimension of asymptotically simple geodesics}\label{sec:dimhausf}

In this section we will prove Theorem \ref{teo:el-wan}. To do so, we need to extend the result from Birman-Series \cite{birman1985geodesics} on the Hausdorff dimension of simple. In that paper, the authors show that if $\mathscr{G}_s$ denotes the set of simple geodesics, then  $\hffdim \mathscr{G}_s = 0$. We prove the stronger result that $\packdim \mathscr{G}_s = 0$, where $\packdim$ refers to the packing dimension (defined in \ref{defi:dimension}) although the argument is essentially the same.

Throughout this section we primarily follow the definitions and notations from Birman-Series \cite{birman1985geodesics} for computations in the hyperbolic plane and Falconer's book \cite{falconer2013fractal} for computations on dimension.

\subsection{Notions of dimension and their properties}\label{subsec:dimensions}

We first define the main two notions of dimension we use to prove Proposition \ref{prop:terminanensimple}, which computes the Hausdorff dimension of asymptotically simple geodesics. In this subsection, $X$ will be a compact metric space. In this paper we always take $X = \partial\HH^2 \times \partial \HH^2$.

Let $A \subset X$. For $\epsilon > 0$, let $N(A, \epsilon)$ denote the smallest number of sets of diameter at most $\epsilon$ needed to cover $A$. 

\begin{defi}\label{defi:dimension}
    Let $A \subseteq X$
    
    \begin{enumerate}
        \item The \emph{upper box dimension} (also called the \emph{upper Minkowski dimension}) of $A$ is defined as
        \[
        \upboxdim A = \limsup_{\delta \to 0} \frac{\log N(A, \epsilon)}{-\log \epsilon}.
        \]
        
        \item The \emph{packing dimension} of $A$ is defined as
        \[
        \packdim A = \inf \left\{ \sup_i \upboxdim A_i : A \subset \bigcup_{i=1}^\infty A_i \right\},
        \]
        where the infimum is taken over all countable coverings of $A$ by sets $A_i$.
        
    \end{enumerate}
\end{defi}

\begin{rmk}\label{rmk:packingbox}
    Just by the definition it is clear that if $A \subseteq X$ then
    $$\packdim A \le \upboxdim A.$$
\end{rmk}

The upper box dimension is a well-known notion of dimension. It has the advantage that it is relatively easy to compute, and it provides an upper bound for the Hausdorff dimension. However, this dimension has the property that
$$\upboxdim A = \upboxdim \bar{A}.$$
This poses some problems: the upper box dimension of a dense subset is equal to the total dimension of the ambient space. Note that the action of $\pi_1(S)$ is minimal in $\mathscr{G}$ (i.e. every orbit is dense). Consequently, the upper box dimension cannot distinguish between different subsets of geodesics: they all end up having full box dimension simply because non-empty $\pi_1(S)$-invariant sets are dense.

In order to solve this issue, we recur to the packing dimension, also known as the modified upper box dimension. It has the following useful property.

\begin{prop}\label{prop:countableunion}
    Let $(A_n)_{n \in \NN}$ be a family of sets on a compact metric space $X$. Then
    $$\packdim \left(\bigcup_{n \in \NN} A_n\right) = \sup_{n \in \NN} \packdim A_n.$$
\end{prop}

The following theorem shows the usefulness of the packing dimension in our setting.

\begin{thm*}[Theorem 3 of Tricot \cite{Tricot1982TwoDO}]
    Let $X$ be a compact metric space, and let $A, B$ subsets of $X$. Then
    $$\hffdim A \times B \le \packdim A + \hffdim B.$$
\end{thm*}

This is not true for the Hausdorff dimension: there are sets $A, B$ such that $\hffdim A = \hffdim B = 0$ but $\hffdim (A \times B) > 0$. For one such construction, one can check Example 7.8 of Falconer's book \cite{falconer2013fractal}.

\subsection{Hausdorff dimension of asymptotically simple geodesics.}

Let $\simpends$, as before, be the set of asymptotically simple geodesics (see Definition \ref{def:ends-simply}). A consequence of Lemma \ref{lem:a-la-estable} is that $\simpends = \pi(\mathscr{G}_s) \times \pi(\mathscr{G}_s)$, where $\mathscr{G}_s$ is the set of simple geodesics. We now state the following proposition:

\begin{prop}\label{prop:terminanensimple}
    The set $\simpends$ of asymptotically simple geodesics has Hausdorff dimension $0$.
\end{prop}

In order to prove this proposition, we will use Theorem 3 of Tricot \cite{Tricot1982TwoDO}, stated in Subsection \ref{subsec:dimensions}. To compute $\packdim \pi(\mathscr{G}_s)$, we will pick representatives of simple geodesics and compute the upper box dimension for the set of such representatives.

\begin{defi}
Let $R$ be a fundamental domain for $S$ in $\HH^2$. A simple geodesic $\gamma$ in $S$ is said to be \textit{non-exceptional with respect to $R$} if every lift $\tilde{\gamma}$ of $\gamma$ that intersects $R$ non-trivially avoids the vertices of $R$.
\end{defi}

Given a fundamental domain $R$, we will consider, for every non-exceptional geodesic, one representative that intersects $R$. Then, we will compute the upper box dimension of such set. Even if we miss some geodesics, the following remark tells us that with two fundamental domains we can have a non-exceptional representative of all but finite geodesics.

\begin{rmk} \label{rem:fundomains}
Let $R_1$ and $R_2$ be finite-sided fundamental domains with no shared vertex. Then only finitely many simple geodesics are exceptional with respect to both $R_1$ and $R_2$.
\end{rmk}

\begin{proof}
    An exceptional geodesic with respect to $R_1$ and $R_2$ needs to be the projection of a geodesic that joins a vertex of $R_1$ with one of $R_2$. Since there are only a finite number of vertices in each fundamental domain, there is only a finite number of geodesics that can be exceptional.
\end{proof}

The proof begins with the following lemma, which is an adaptation of Proposition 4.1 from  Birman-Series \cite{birman1985geodesics}.

\begin{lem}\label{lemma:goodcover}
    Let $A_R$ be the set of non-exceptional geodesics for a given fundamental domain $R$. Then there exist constants $\eta, c > 0$, depending only on $R$, and a polynomial $P$ such that, for each $n \in \mathbb{N}$, there exists a set $\tilde{F}_n \subseteq \mathscr{G}_s$ satisfying $\lvert \tilde{F}_n \rvert < P(n)$, with the following property:

    For every $\gamma \in A_R$, and for every lift $\tilde{\gamma}$ that intersects $R$, there exists $g \in \tilde{F}_n$ such that
    \[
    d(\tilde{\gamma}, g) < c e^{-\eta n}.
    \]
\end{lem}

\begin{proof}
    According to Proposition 4.1 in Birman-Series \cite{birman1985geodesics} and its proof, there exists a polynomial $P$ with the following property: Given $n \in \NN$, there exists a set $F_n$ of geodesic arcs with endpoints in $\partial R$, with $\text{card}(F_n) < P(n)$ such that for every geodesic $\gamma \in A_R$ and every lift $\tilde{\gamma}$ that intersects $R$ there is a geodesic arc $\delta \in F_n$ whose completion lies in the same $n$ consecutive fundamental domains as $\tilde{\gamma}$, both in the future and the past. 

    Define the set $$\tilde{F}_n = \{ (\delta^-, \delta^+): \delta \text{ is a geodesic with an arc that belongs to } F_n \}. $$
Clearly this set is in bijection with $F_n$. Thus, we only need to show the existence of the constants $c$
 and $\eta$ to get the desired result.


Let $(\gamma^-, \gamma^+) \in A_R$. By construction, there exists $\delta \in F_n$ and $t_0, t'_0, t_n, t'_n \in \RR$ such that:
\begin{itemize}
\item $ t_n \geq n\eta $ and $ t'_n \geq n\eta $,
\item The segment $ \gamma|_{[t_0, t_n]} $ remains within a $\operatorname{diam}(R)$-neighborhood of $ \delta|_{[t'_0, t'_n]} $.
\end{itemize}
Lemma 1.1(a) of Birman-Series \cite{birman1985geodesics} states that there exists a constant $ c_1 $ such that

\[
d(\gamma_{t_n}, \delta_{t'_n}) < c_1 e^{-\eta n}.
\]

Applying part (b) of the same lemma, we also obtain constants $ c_2 $ with

\[
d(\gamma_{t_n}, \gamma^+) < c_2 e^{-\eta n}, \quad d(\delta_{t'_n}, \delta^+) < c_2 e^{-\eta n}.
\]

By the triangle inequality, it follows that:

\[
d(\gamma^+, \delta^+) < (c_1 + c_2) e^{-\eta n}, \quad d(\gamma^-, \delta^-) < (c_1 + c_2) e^{-\eta n}.
\]

Setting $ c = 2(c_1 + c_2) $, we obtain the desired result.
\end{proof}

\begin{cor}
    Let $R$ be a fundamental domain. There exists a set $A'_R \subseteq A_R$ such that $A'_R$ projects to the same set in $S$ as $A_R$ and $\upboxdim A'_R = 0$.
\end{cor}

\begin{proof}
    For each $\gamma$ in $A_R$ pick a lift $\tilde{\gamma}$ that goes through $R$. Given $n \in \NN$ Using the constants of Lemma \ref{lemma:goodcover}, if $a_n = ce^{-\eta n}$ there exists an element $g \in \tilde{F_n}$ such that $d(\tilde{\gamma}, g) < a_n$. Thus, $N(\mathscr{G}_s, a_n) \le P(n)$. Now let $b_n$ be a monotone sequence such that $b_n \to 0$. Then there exists a sequence $n_k$ such that $b_n \in [a_{n_k}, a_{n_{k + 1}}]$. Thus, we get the following chain of inequalities.

    $$\frac{N(A_R, b_n)}{-\log b_n} \le \frac{N(A_R, a_{n_k+1})}{-\log a_{n_k}} \le \frac{\log P(n_k + 1)}{- \log(c_2) + \eta n_k}.$$

    Because the limit of the expression on the right is $0$ we get the desired result.
\end{proof}

We have the following consequence of the Corollary.

\begin{cor}\label{cor:packdimsimple}
    We have that $\packdim \mathscr{G}_s = 0$.
\end{cor}

\begin{proof}
    Let $R_1$ and $R_2$ be two fundamental domains with no vertex in common. Remark \ref{rem:fundomains} shows that there exists a finite set $F \subseteq \mathscr{G}_s$ such that for every simple geodesic $\gamma$ there is a lift $\tilde{\gamma} \in A'_{R_1} \cup A'_{R_2} \cup F$. The upper box dimension is finitely stable, thus
    $$\upboxdim (A'_{R_1} \cup A'_{R_2} \cup F) = \max\left\{ \upboxdim(A'_{R_1}), \upboxdim(A'_{R_2}), F \right\} = 0$$
    and because of Remark \ref{rmk:packingbox} we know
    $$\packdim (A'_{R_1} \cup A'_{R_2} \cup F) = 0.$$
    By definition, it is true that
    $$\mathscr{G}_s = \bigcup_{\gamma \in \Gamma} \gamma \cdot (A'_{R_1} \cup A'_{R_2} \cup F)$$
    and because of Proposition \ref{prop:countableunion} we also have that 
    $$\packdim \mathscr{G}_s = 0.$$
\end{proof}

Proposition \ref{prop:terminanensimple} follows as a consequence of Corollary \ref{cor:packdimsimple}.

\begin{proof}[Proof of Proposition \ref{prop:terminanensimple}]
    Because projections are Lipschitz maps, we get that 
    $$\packdim \pi(\mathscr{G}_s) = \hffdim \pi(\mathscr{G}_s) = 0.$$
    Because of Theorem 3 of Tricot, stated in Subsection \ref{subsec:dimensions} \cite{Tricot1982TwoDO} we have that
    $$\hffdim \pi(\mathscr{G}_s) \times \pi(\mathscr{G}_s) \le \hffdim \pi(\mathscr{G}_s) + \packdim \pi(\mathscr{G}_s) = 0,$$
    and the Proposition is proven.
\end{proof}

We are now ready to prove the first theorem of the paper, which we restate for convenience. Recall that $\boundangle$ is the set of geodesics whose angles of self-intersection are all greater or equal than $\delta$.

\begin{thm*}[Theorem 1]
    For all $\delta > 0$, we have $\hffdim \boundangle = 0$. In particular, the typical geodesic has arbitrarily small self-intersection angles. Moreover, the image of $\boundangle$ as a subset of $S$ has Hausdorff dimension 1.
\end{thm*}

\begin{proof}[Proof of Theorem \ref{teo:el-wan}]
    We need to prove that $\hffdim \boundangle = 0$. Observe that if $p: S' \to S$ is a covering between hyperbolic surfaces, it induces an equivariant diffeomorphism of $\mathscr{G}$, and thus sends $\pi_1(S')$ invariant sets into $\pi_1(S)$ invariant subsets.

    Because finite covers of $S$ up to covering equivalence are in bijection with finite index subgroups of $\pi_1(S)$, and finite index subgroups are finitely generated, the set of finite covers (up to covering equivalence) is countable; we denote the set of such covers by $\{(S_n, p_n)\}$. Let $\simpends(S_n)$ be the set of asymptotically simple geodesics of $S_n$. Proposition \ref{prop:simple-ends} tells us that
    $$\boundangle \subseteq \bigcup_{n \in N} p_n(\simpends(S_n)).$$

    Observe that, because of Proposition \ref{prop:terminanensimple}, we get $\hffdim \simpends(S_n) = 0$, and this implies that $\hffdim p_n(\simpends(S_n)) = 0$. Finally, we have that
    $$\hffdim \boundangle = \hffdim \left(\bigcup_{n \in N} p_n(\simpends(S_n))\right) = \sup \left(\hffdim p_n(\simpends(S_n))\right) = 0$$
    which proves the first part of the theorem.

    The second part of the theorem is a consequence of the Hopf parametrization: a diffeomorphism 
    $$h: T^1\HH^2 \to \mathcal{G} \times \RR^2.$$
    Let $Im \boundangle$ be the projection of $\boundangle$ to the surface. Observe that $\hffdim (Im \boundangle) \ge 1$, since it contains geodesic arcs. Consider $A$ the image of $\boundangle$ in $T^1S$, we will instead prove that $A$ has Hausdorff dimension 1, because the projection to the surface is a Lipschitz map, the theorem follows.

    Assume by contradiction that $\hffdim A > 1$. Because $T^1S$ is compact, there exists a small neighborhood $U \subseteq T^1S$, diffeomorphic to $V \times \Ss^1 \subseteq T^1\HH^2$, with 
    $\hffdim \left(U \cap A\right) > 1$.

    Consider $\tilde{U}$, $\tilde{A}$ lifts to $T^1\HH^2$ of $U$ and $U \cap A$ respectively, such that the restriction of the projection to $T^1S$ is a diffeomorphism onto $U$. Take the set
    $$\mathscr{A} = \{(x, y) \in \mathscr{G}: \text{ there exists } t \in \RR \text{ with } h^{-1}(x, y, t) \in \tilde{A} \},$$
    and $K > 0$ with $\tilde{A} \subseteq h^{-1}(\mathscr{A} \times [-K, K]).$

    Because $\hffdim \boundangle = 0$, we get that $\hffdim \mathscr{A} = 0$. Moreover, we have the following chain of inequalities:
    $$
    \begin{aligned}
    \hffdim \tilde{A} &\le \hffdim h^{-1}(\mathscr{A} \times [-K, K]) = \hffdim (\mathscr{A} \times [-K, K])\\
    &= \hffdim(\mathscr{A}) + \upboxdim([-K, K]) = \hffdim(\mathscr{A}) + 1,
    \end{aligned}
    $$
    which contradicts the assumption.
\end{proof}

%% file: no_triangles.tex
\section{Geodesics that do not bound triangles}

To conclude the paper, we prove the Theorem \ref{teo:triangulos}. The idea of the proof is to reduce to the case of Theorem \ref{teo:el-wan}. Thus, we will prove that $\notboundtri$, the set of geodesics that do not bound a triangle, is a subset of $\bigcup_{\delta > 0}\boundangle$. Recall that a geodesic $\alpha$ \emph{bounds a triangle} if there exist lifts $\alpha_1, \alpha_2, \alpha_3$ that pairwise intersect. In order to prove the theorem we will show that if the angles of self-intersection of a given geodesic go to 0, then such geodesic lies in the complement of a geodesic lamination. The theorem then follows from the geometry of such subsets of the surface.

\subsection{The complement of a geodesic lamination}
We will briefly review results from Casson-Bleiler's book \cite{casson1988automorphisms}.
There are two shapes that the complement of a geodesic lamination with no closed leaves can have. The first one is an ideal polygon, and the second one is related to crowns. All this is explored in Chapter 4 of the book.

\begin{defi}
\begin{enumerate}
    \item An $n$-sided \textit{ideal polygon}  is a surface isometric to the convex core of $n$ points in $\partial \HH^2$ 
    \item A \textit{crown} is a complete hyperbolic surface $C$ of finite area and geodesic boundary, which is homeomorphic to $S^1 \times [0,1] \setminus A$, where $A$ is a finite set contained in $S^1 \times \{1\}$.
\end{enumerate}
\end{defi}

\begin{lem}[Lemma 4.4 of Casson and Bleiler \cite{casson1988automorphisms}]\label{lem:Casson-Bleiler}
    Let $\lambda$ be a geodesic lamination without closed leaves on a closed hyperbolic surface $S$. If $U$ is a connected component of $S \setminus \lambda$ then either $U$ is an ideal polygon, or there is a unique compact subset $U_0$ of $U$ such that $U \setminus U_0$ is isometric to a finite disjoint union of interiors of crowns. 
\end{lem}

Now that we have the last ingredients, we are ready to prove the theorem.

\subsection{Proof of theorem \ref{teo:triangulos}}




Before proving Theorem \ref{teo:triangulos}, we need to prove the following lemma on hyperbolic geometry. It states that any simple geodesic that gets arbitrarily close to a simple loop spirals around that loop.

\begin{lem}\label{lem:hojacerrada}
    Let $\alpha$ be a simple geodesic. Assume that there exists a sequence $\{t_n\}_{n \in \NN}$ going to $+\infty$ such that $\dot{\alpha}(t_n)$ converges in $T^1S$ to a point in a closed geodesic $\dot{\gamma}$. Then there exist lifts $\tilde{\alpha}, \tilde{\gamma}$ to $\HH^2$ of $\alpha$ and $\gamma$ respectively such that either $\tilde{\alpha}(+\infty) = \tilde{\gamma}(+\infty)$ or $\tilde{\alpha}(+\infty) = \tilde{\gamma}(-\infty)$. 

\end{lem}

\begin{rmk}
    This lemma can be stated in dynamical terms: If $\alpha$ is simple and $\omega^+(\alpha)$ contains a closed curve $\gamma$, then $\omega^+(\alpha) = \gamma$.
\end{rmk}

\begin{proof}
    Take a lift $\tilde{\gamma}$ of $\gamma$ to $\HH^2$, and $D$ a fundamental domain of the action of $\pi_1(S)$ in $\HH^2$. Consider $\tilde{\alpha}_n$ lifts of $\alpha$ such that $\tilde{\alpha}_n(t_n)$ converges to a point $p$ in $\tilde{\gamma}$.
    
    Assume by contradiction that, for every $n \in \NN$, we have $\alpha_n(+\infty) \neq \gamma(+\infty)$ and $\alpha_n(+\infty) \neq \gamma(-\infty)$. Assume without loss of generality (up to reversing the orientation of $\gamma$) that the points $\gamma(-\infty), \alpha_n(-\infty)$, $\alpha_n(+\infty), \gamma(+\infty)$ are cyclically ordered. 
    
    Let $g$ be the element in $\pi_1(S)$ preserving $\gamma$, and let $[x, gx]$ be a fundamental domain for the action of $g$ in the component of $\partial\HH^2 \setminus \{\gamma^-, \gamma^+\}$ that contains $\tilde{\alpha}_n(-\infty)$ and $\tilde{\alpha}_n(+\infty)$. Observe that $\tilde{\alpha}_n(+\infty)$ and $\tilde{\alpha}_n(-\infty)$ converge to $\gamma^+$ and $\gamma^-$ respectively. Thus, there exists $n$ big enough such that $\tilde{\alpha}_n(+\infty), \tilde{\alpha}_n(-\infty) \notin [x, gx]$. If we take $k \in \NN$ such that $g^k\tilde{\alpha}_n(+\infty) \in [x, gx]$, we will get that $g^k\tilde{\alpha}_n$ intersects $\tilde{\alpha}_n$. Which is a contradiction with the fact that $\tilde{\alpha}$ is simple
\end{proof}

Now we are ready to prove the theorem, which we restate for convenience.

\begin{thm*}[Theorem \ref{teo:triangulos}]
    The set $\mathscr{G}_\Delta$ of geodesics that bound a triangle is typical. Moreover, the image of the complement of $\mathscr{G}_\Delta$ in $S$ has Hausdorff dimension $1$.
\end{thm*}

\begin{proof}[Proof of Theorem \ref{teo:triangulos}]
    The strategy is to prove, by contradiction, that $\notboundtri \subseteq \bigcup_{\delta > 0} \boundangle$. By monotonicity of the Hausdorff dimension, this proves both statements of the theorem.

    Let $\alpha \in \notboundtri$. Suppose that there is no $\delta > 0$ such that $\alpha \in \boundangle$. This means that either:
    \begin{enumerate}
        \item\label{item:angulo-0} There exist sequences $t_n, s_n$ in $\RR$ such that $\alpha(t_n) = \alpha(s_n)$ and $\angle (\dot{\alpha}(t_n), \dot{\alpha}(s_n)) \to 0$.
        \item\label{item:angulo-pi} There exist sequences $t_n, s_n$ in $\RR$ such that $\alpha(t_n) = \alpha(s_n)$ and $\angle (\dot{\alpha}(t_n), \dot{\alpha}(s_n)) \to \pi$.
    \end{enumerate}


We will treat both cases analogously. Because $S$ is compact, up to taking a subsequence we may assume that $\dot{\alpha}(t_n)$ converges to a point $\dot{p} = (p, v) \in T^1S$. Observe that this also means that $\dot{\alpha}(s_n)$ converges to $\dot{p}$ in the first case and that $\dot{\alpha}(s_n)$ converges to $(p, -v)$ in the second case.

\begin{claim}
    The geodesic through $\dot{p}$ (which we call $\gamma$) is simple and does not intersect $\alpha$.
\end{claim}
\begin{proof}[Proof of Claim 1]
In order to prove this claim, we go to the universal cover. Take lifts $\dot{p}_n, \dot{q}_n$ and $\tilde{\gamma}$ of $\dot{\alpha}(t_n)$, $\dot{\alpha}(s_n)$ and $\gamma$ respectively. For each $\dot{p} \in T^1\HH^2$, we denote by $\alpha_{\dot{p}}$ the geodesic tangent to $\dot{p}$.

Assume by contradiction that $\alpha$ intersects $\gamma$, then there exists a lift $\tilde{\alpha}$ of $\alpha$ that intersects $\tilde{\gamma}$. Observe that $\alpha_{\dot{p}_n}(+\infty)$ converges to $\tilde{\gamma}(+\infty)$ and $\alpha_{\dot{p}_n}(-\infty)$ converges to $\tilde{\gamma}(-\infty)$. Because two geodesics intersecting is an open property in the set of pairs of points in the boundary, we get that for $n$ big enough $\alpha_{\dot{q}_n}$ and $\alpha_{\dot{p}_n}$ both will intersect $\tilde{\alpha}$, which is a contradiction with the fact that $\alpha$ does not bound a triangle. A similar argument shows that $\gamma$ must be simple, which proves the claim.
\end{proof}


Now let $\lambda = \overline{\gamma}$. Because of Claim 1, we have that $\lambda$ is a lamination. Moreover, a similar argument to the proof of Claim 1 shows $\alpha$ cannot intersect any recurrent leaf of $\lambda$.

\begin{claim}
    $\lambda$ cannot contain a closed leaf. 
\end{claim} 
\begin{proof}[Proof of Claim 2]
See Figure \ref{fig:argumento-no-triangulos} for a visual representation of the argument. Suppose by contradiction that $\lambda$ contains a closed leaf $\beta$. We assume $\beta \neq \gamma$ (a simpler but similar argument works when $\beta = \gamma$). Because of Lemma \ref{lem:hojacerrada}, up to reversing the orientation of $\gamma$, we may assume that there exists a lift $\tilde{\beta}$ of $\beta$ such that $\tilde{\gamma}(+\infty) = \tilde{\beta}(+\infty)$. Denote by $g \in \pi_1(S)$ the element that preserves $\tilde{\beta}$.  Since $\tilde{\beta}$ cannot intersect any lift of $\alpha$, it is true that $\tilde{\beta}(-\infty)$ is in the component of $\partial\HH^2 \setminus \{\tilde{\gamma}(+\infty), \tilde{\gamma}(-\infty)\}$ that does not contain the endpoints of $\alpha_{\dot{p}_n}$. Observe that $\alpha_{\dot{p}_n}(+\infty)$, $\alpha_{\dot{p}_n}(-\infty)$ converge to $\tilde{\gamma}(+\infty) = \tilde{\beta}(+\infty)$ and $\tilde{\gamma}(-\infty)$ respectively. 

Because $\pi_1(S)$ acts continuously in the boundary, it is true that $g^{-1}\alpha_{\dot{p}_n}(-\infty)$ converges to $g^{-1}\tilde{\gamma}(-\infty)$, while $g^{-1}\alpha_{\dot{p}_n}(+\infty)$ converges to $\tilde{\beta}(+\infty)$. This implies that for $n$ big enough, these two points are in different components of $\partial\HH^2 \setminus \{\tilde{\gamma}(+\infty), \tilde{\gamma}(-\infty)\}$, which contradicts the fact that $\gamma$ and $\alpha$ do not intersect.

\begin{figure}
    \centering
    \def\svgwidth{0.8\columnwidth}
    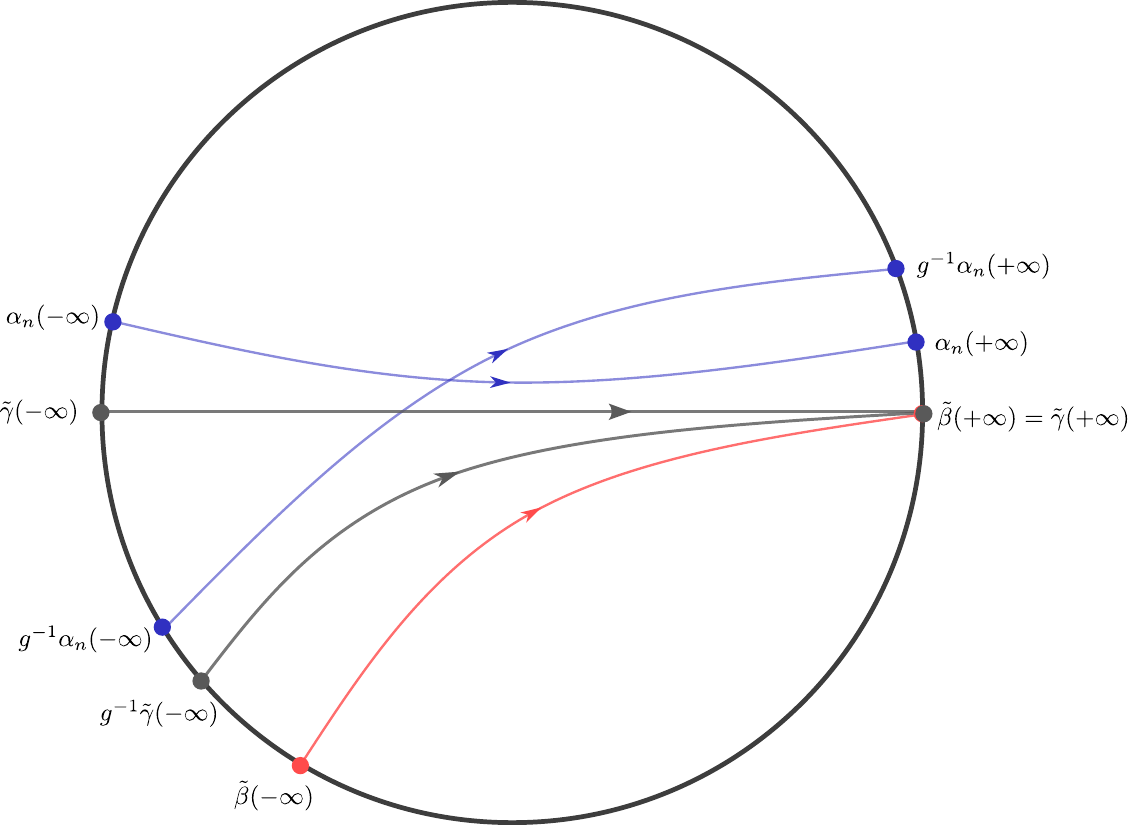
    \caption{A visualization of the main argument of Case 2 in the proof of Theorem \ref{teo:triangulos}.}
    \label{fig:argumento-no-triangulos}
  \end{figure}


\end{proof}
Because of Claims 1 and 2, we know that $\alpha$ lies in a connected component $U$ of $S \setminus \lambda$, where $\lambda$ is a lamination without closed leaves. Because of Lemma \ref{lem:Casson-Bleiler} we have that either $U$ is a finitely sided polygon with ideal vertices, or there is a unique compact surface with geodesic boundary $U_0$ such that $U \setminus U_0$ is a finite union of crowns. The case where $U$ is a finitely sided polygon is trivial because complete geodesics contained in polygons must be simple. So we assume $U$ is of the latter form.

The compact surface $U_0$ lies at positive distance of the lamination, so there exists $t > 0$ and a crown $C$ bounded by a closed geodesic $\beta$ in $U$ such that $\alpha(t) \in C$. Now there are two possibilities, either $\alpha(t) \in C$ for every $t \in \RR$ or there exists $s \in \RR$ such that $\alpha(s) \in \beta$. The first case is impossible, because a geodesic contained in $C$ cannot have an infinite number of self-intersections. Now assume the second condition is not true, and assume without loss of generality that $s < t$ (the other case is analogous by reversing the orientation of $\alpha$).

Observe that, again by the geometry of the crown, there cannot exist $t' \in \RR$ with $t > t'$ such that $\alpha(t') \in \beta$ (i.e. the geodesic cannot exit $C$ through $\beta$). Because it also cannot intersect $\lambda$, this means that $\alpha$ must end at an ideal point of the crown. An analogous analysis shows that $\alpha$ must start at an ideal point of some crown, and both of these facts are in contradiction with the fact that $\alpha$ has an infinite number of self-intersections.

The proof of the second statement is analogous to that of Theorem \ref{teo:el-wan}.

\end{proof}

%% file: argumento-no-triangulos.pdf_tex
\begingroup%
  \makeatletter%
  \providecommand\color[2][]{%
    \errmessage{(Inkscape) Color is used for the text in Inkscape, but the package 'color.sty' is not loaded}%
    \renewcommand\color[2][]{}%
  }%
  \providecommand\transparent[1]{%
    \errmessage{(Inkscape) Transparency is used (non-zero) for the text in Inkscape, but the package 'transparent.sty' is not loaded}%
    \renewcommand\transparent[1]{}%
  }%
  \providecommand\rotatebox[2]{#2}%
  \newcommand*\fsize{\dimexpr\f@size pt\relax}%
  \newcommand*\lineheight[1]{\fontsize{\fsize}{#1\fsize}\selectfont}%
  \ifx\svgwidth\undefined%
    \setlength{\unitlength}{540.49429826bp}%
    \ifx\svgscale\undefined%
      \relax%
    \else%
      \setlength{\unitlength}{\unitlength * \real{\svgscale}}%
    \fi%
  \else%
    \setlength{\unitlength}{\svgwidth}%
  \fi%
  \global\let\svgwidth\undefined%
  \global\let\svgscale\undefined%
  \makeatother%
  \begin{picture}(1,0.73235456)%
    \lineheight{1}%
    \setlength\tabcolsep{0pt}%
    \put(0,0){\includegraphics[width=\unitlength,page=1]{argumento-no-triangulos.pdf}}%
  \end{picture}%
\endgroup%